\documentclass[11pt]{amsart}

\usepackage[T1]{fontenc}
\usepackage{lmodern}
\usepackage{amsmath,amssymb,mathtools}
\usepackage{microtype}
\usepackage[hidelinks]{hyperref}
\hypersetup{
  pdftitle={sharp spectral constanta for scaled q-numerical ranges},
  pdfauthor={Mohamed Amine Aouichaoui and Ryan O'Loughlin},
  pdfsubject={A sharp transfer principle for scaled q-numerical ranges},
  pdfkeywords={q-numerical range, spectral set, Crouzeix conjecture, similarity}
}
\usepackage[nameinlink,noabbrev]{cleveref}

\numberwithin{equation}{section}

\newtheorem{theorem}{Theorem}[section]
\newtheorem{proposition}[theorem]{Proposition}
\newtheorem{lemma}[theorem]{Lemma}
\newtheorem{corollary}[theorem]{Corollary}

\crefname{proposition}{proposition}{propositions}
\Crefname{proposition}{Proposition}{Propositions}

\newcommand{\C}{\mathbb C}

\newcommand{\Mn}{M_n(\C)}

\newcommand{\norm}[1]{\left\lVert #1\right\rVert}
\newcommand{\abs}[1]{\left\lvert #1\right\rvert}
\newcommand{\ip}[2]{\left\langle #1,#2\right\rangle}
\newcommand{\spn}{\operatorname{span}}
\newcommand{\Int}{\operatorname{int}}

\title[Sharp spectral constants for scaled $q$-numerical ranges]
{Sharp spectral constants for scaled $q$-numerical ranges}

\author{Mohamed Amine Aouichaoui}
\address{Department of Mathematics, Faculty of Sciences of Monastir, University of Monastir, 5019 Monastir, Tunisia}
\email{aouichaoui.mohamedamine@fsm.u-monastir.tn}
\email{amine.aouichaoui@fsm.rnu.tn}

\author{Ryan O'Loughlin}

\address{Department of Mathematics and Statistics, University of Reading, Reading, England}
\email{r.d.oloughlin@reading.ac.uk}

\subjclass[2020]{Primary 47A12, 47A25; Secondary 15A60}
\keywords{$q$-numerical range, numerical range, spectral set, Crouzeix conjecture, similarity, Kantorovich inequality}

\begin{document}

\begin{abstract}
For $ n \geq 2$, $A\in M_n(\mathbb C)$ and $0<|q|\leq 1$, let
$\Omega_q(A)=q^{-1}W_q(A)$ be the scaled $q$-numerical range. We prove that for every $\gamma \geq 1$, 
\[
 \Omega_{\eta(\gamma)}(A)
 =\bigcup_{\kappa(S)\leq\gamma}W(S^{-1}AS),
 \qquad
 \eta(\gamma)=\frac{2}{\gamma+\gamma^{-1}},
\]
where $\kappa(S)=\|S\|\,\|S^{-1}\|$. As a consequence, we prove the sharp inequality
\[
 \|p(A)\|\leq
 \max\!\left\{1,\frac{2|q|}{1+\sqrt{1-|q|^2}}\right\}
 \max_{z\in\Omega_q(A)}|p(z)|,
\]
for all polynomials $p$.
\end{abstract}

\maketitle

\section{Introduction}

Throughout, $n\geq2$. Let $A\in M_n(\mathbb C)$ and let $K\subset\mathbb C$ be compact with
$\sigma(A)\subset K$.  For $C\geq1$, we say that $K$ is a
$C$-spectral set for $A$ if
\[
 \|f(A)\|\leq C\max_{z\in K}|f(z)|
\]
for every rational function $f$ with poles off $K$. Equivalently, the same inequality holds for every function holomorphic on a neighbourhood of $K$. If $\mathbb C\setminus K$ is connected, Runge’s theorem shows that it is enough to consider polynomials. Any such $C$ is called a spectral constant for $A$ on $K$.  The least admissible value is called the optimal spectral constant.

For a matrix $A\in\Mn$, its numerical range is
\[
 W(A)=\{\ip{Ax}{x}:x\in\C^n,\ \norm{x}=1\}.
\]
Crouzeix conjectured that the numerical range $W(A)$ is a
$2$-spectral set for every square matrix $A$, i.e.
\begin{equation}\label{eq:Crouzeix-conjecture}
 \norm{p(A)}\leq 2\max_{z\in W(A)}\abs{p(z)}
\end{equation}
for every complex polynomial $p$; see \cite{Crouzeix2004}. Crouzeix's Conjecture was recently proved in \cite{Jin} (and also see \cite{LoristSchwenninger2026}). Prior to these developments, Crouzeix established a dimension-free spectral
constant in \cite{Crouzeix2007}, and Crouzeix and Palencia proved the numerical range was a $(1 + \sqrt{2})$-spectral set (see \cite{CrouzeixPalencia2017} and the short reformulation in
\cite{RansfordSchwenninger2018}). We refer to \cite{BickelEtAl2020} for a
survey of the conjecture and its extremal-function theory.

The $q$-numerical range originated in the study of constrained bilinear
forms \cite{MarcusAndresen1977}.  With the convention that the inner product
is linear in the first variable, it is defined, for $0<|q|\leq1$, by
\[
 W_q(A)=\{\ip{Ax}{y}:\norm{x}=\norm{y}=1,\ \ip{x}{y}=q\}.
\]
Its convexity was proved by Tsing \cite{Tsing1984}.  Li, Mehta, and Rodman
subsequently developed its basic geometry, norm properties, and dependence
on $q$ \cite{LiMehtaRodman1994}.  Since $W_q(A)$ itself need not contain the
spectrum, the scaled set
\begin{equation}\label{eq:scaled-range}
 \Omega_q(A)=\frac1qW_q(A)
\end{equation}
was introduced in \cite{OLoughlinRani2026}, where spectral-set estimates
were obtained and $\Omega_q(A)$ was conjectured to be a $\max\!\left\{1,\frac{2|q|}{1+\sqrt{1-|q|^2}}\right\}$-spectral set. Explicitly, \cite[Conjecture~4.2]{OLoughlinRani2026} conjectures that
\begin{equation}\label{eq:q-conjecture}
 \norm{p(A)}\leq
 \max\!\left\{1,\frac{2|q|}{1+\sqrt{1-|q|^2}}\right\}
 \max_{z\in\Omega_q(A)}\abs{p(z)}, \quad \text{ for all polynomials } p .
\end{equation}

The main result of this paper is Theorem \ref{thm:abstract-transfer}, which shows the following. If for every matrix $B$, the compact set $\mathcal K(B) \subseteq \mathbb{C}$ is a
$C$-spectral set for $B$ and if $D$ is a bounded convex domain containing
$\sigma(A)$ and, for some $\gamma>1$,
\[
 \mathcal K(S^{-1}AS)\Subset D
 \qquad\text{whenever}\qquad
 \kappa(S)<\gamma,
\]
then $\overline{D}$ is a $\max\left\{1,\frac{C}{\gamma}\right\}$-spectral set. Combining Corollary \ref{cor:similarity-interior} with Theorem \ref{thm:abstract-transfer} with $\mathcal K(B)=W(B)$ and
$D=\Int\Omega_q(A)$ to show that if $W(B)$ is a $C$-spectral set for
every  $B$, then $\Omega_q(A)$ is a
\[
 \max\!\left\{1,\frac{C|q|}{1+\sqrt{1-|q|^2}}\right\}\text{-spectral set}.
\]
Consequently, since numerical ranges are 2-spectral sets, taking $C=2$ establishes \eqref{eq:q-conjecture}.

A key geometric ingredient, which may be of independent interest, is the
following exact similarity formula, proved in
Section \ref{2}. For $\gamma\geq1$ 
\[
 \Omega_{\eta(\gamma)}(A)
 =\bigcup_{\kappa(S)\leq\gamma}W(S^{-1}AS),
 \qquad
 \eta(\gamma)=\frac{2}{\gamma+\gamma^{-1}}.
\]
During the final stages of preparation of this manuscript Jin  updated his preprint \cite{Jin} with a solution of Crouzeix's Conjecture to also include a proof of \eqref{eq:q-conjecture}. In comparison to Jin's results, our Theorem \ref{thm:similarity-amplification}, and Theorem \ref{thm:abstract-transfer} apply not only to scaled (q)-numerical ranges but they give a general mechanism for transferring spectral-set estimates through similarities.

% A recent preprint of Jin \cite{Jin} announces a proof of
% Crouzeix's Conjecture. The implication established here is
% independent of the validity of that announcement.  If the announced proof is
% confirmed (or if Crouzeix's conjecture is established by any other
% method) then Theorem \ref{thm:transfer}  gives the conjectured
% constant \eqref{eq:q-conjecture} for scaled $q$-numerical ranges.

\section{Scaled \texorpdfstring{$q$}{q}-numerical ranges and similarities}\label{2}

Throughout, $n\geq2$.  Vector norms are Euclidean and matrix norms are the
induced operator norms.  For a compact set $K\subset\C$, write
\[
 \norm{f}_K=\max_{z\in K}\abs{f(z)}.
\]
For an invertible matrix $S$, the condition number of $S$ is
\[
 \kappa(S)=\norm{S}\,\norm{S^{-1}}.
\]
We first record the elementary reduction from a complex parameter to its
modulus and the nesting property of the scaled ranges.  The non-strict
inclusion in part (2) below also appears in
\cite{OLoughlinRani2026}; for completeness, we include its short proof.

\begin{proposition}\label{prop:nesting}
Let $A\in\Mn$.
\begin{enumerate}
\item If $0<|q|\leq1$, then $\Omega_q(A)=\Omega_{|q|}(A)$.
\item If $0<r<t\leq1$, then
\[
 \Omega_t(A)\subseteq\Omega_r(A).
\]
If $A$ is not scalar, then the inclusion is strict in the stronger form
\[
 \Omega_t(A)\subseteq\Int\Omega_r(A).
\]
\item If $0<r<1$, then
\[
 \sigma(A)\subseteq\Int\Omega_r(A)
\]
if and only if $A$ is not scalar.
\end{enumerate}
\end{proposition}

\begin{proof}
\begin{enumerate}
\item
Write $q=re^{i\theta}$, where $r=|q|$.  If $\ip{x}{y}=r$, then
\[
 \ip{x}{e^{-i\theta}y}=e^{i\theta}r=q,
 \qquad
 \ip{Ax}{e^{-i\theta}y}=e^{i\theta}\ip{Ax}{y}.
\]
It follows that $W_q(A)=e^{i\theta}W_r(A)$ and so dividing by
$q=re^{i\theta}$ proves (1).

\item
Let
$w=\ip{Ax}{y}\in W_t(A)$, with $\norm{x}=\norm{y}=1$ and $\ip{x}{y}=t$.
Write
\[
 y=tx+\sqrt{1-t^2}\,u,
 \qquad u\perp x,\quad \norm{u}=1.
\]
For $\zeta\in\mathbb T$, the vector
$y_\zeta=rx+\zeta\sqrt{1-r^2}\,u$ is a unit vector satisfying
$\ip{x}{y_\zeta}=r$, and
\[
 \ip{Ax}{y_\zeta}
 =r\ip{Ax}{x}
  +\overline\zeta\sqrt{1-r^2}\,\ip{Ax}{u}.
\]
As $\zeta$ runs through $\mathbb T$, these points trace the circle with
centre $r\ip{Ax}{x}$ and radius
$\sqrt{1-r^2}\abs{\ip{Ax}{u}}$.  The convexity of $W_r(A)$ therefore puts
the closed disk bounded by this circle in $W_r(A)$.  Now
\[
 \frac rt w
 =r\ip{Ax}{x}
  +\frac rt\sqrt{1-t^2}\,\ip{Ax}{u},
\]
and
\[
 \frac r t\sqrt{1-t^2}\leq\sqrt{1-r^2},
\]
because the squared inequality is equivalent to $r^2\leq t^2$.
Thus $(r/t)w$ belongs to the disk above.  Equivalently,
$w/t\in\Omega_r(A)$, proving the non-strict inclusion.

The strict interior inclusion for non-scalar matrices is
\cite[Theorem~2.5]{LiMehtaRodman1994}.

\item
\cite[Theorem~2.7]{LiMehtaRodman1994} states that
$\operatorname{conv}(r\sigma(A))\subseteq\Int W_r(A)$ when $A$ is
non-scalar and $r<1$; division by $r$ proves the forward implication.
Conversely, if $A=\lambda I$, then
$\Omega_r(A)=\{\lambda\}$ has empty interior, so it cannot contain
$\sigma(A)=\{\lambda\}$ in its interior.
\end{enumerate}
\end{proof}

For $\gamma\geq1$, define
\begin{equation}\label{eq:eta}
 \eta(\gamma)=\frac{2}{\gamma+\gamma^{-1}}
             =\frac{2\gamma}{\gamma^2+1}.
\end{equation}
The function $\eta$ is decreasing on $[1,\infty)$.  The next result is the geometric core of the paper. 
% Notice that
% $\eta(\gamma)^{-1}=(\gamma+\gamma^{-1})/2$ is also the function which sends a complete spectral-set constant to a corresponding complete numerical-radius
% constant \cite{DavidsonPaulsenWoerdeman2017}.

\begin{theorem}\label{thm:similarity-envelope}
Let $A\in\Mn$ and $\gamma\geq1$.  Then
\begin{equation}\label{eq:similarity-envelope}
 \Omega_{\eta(\gamma)}(A)
 =
 \bigcup_{\substack{S\in\Mn\ {\rm invertible}\\ \kappa(S)\leq\gamma}}
 W(S^{-1}AS).
\end{equation}
% In particular, every invertible $S\in\Mn$ satisfies
% \begin{equation}\label{eq:pointwise-similarity}
%  W(S^{-1}AS)\subseteq\Omega_{\eta(\kappa(S))}(A).
% \end{equation}
\end{theorem}

\begin{proof}
We first prove the inclusion from right to left.  Fix an invertible $S$
with $\kappa(S)\leq\gamma$ and let $u$ be a unit
vector. For ease of notation, set
\[
 a=\norm{Su},\qquad b=\norm{S^{-*}u},\qquad
 x=\frac{Su}{a},\qquad y=\frac{S^{-*}u}{b}.
\]
Then $x$ and $y$ are unit vectors, and
\begin{equation}\label{eq:t-def}
 t:=\ip{x}{y}=\frac1{ab}>0.
\end{equation}
Moreover,
\begin{equation}\label{eq:similarity-point}
 \ip{S^{-1}ASu}{u}=ab\,\ip{Ax}{y}=\frac1t\ip{Ax}{y}\in\Omega_t(A).
\end{equation}

Let $m$ and $M$ be the smallest and largest singular values of $S$,
respectively, and put $P=S^*S$.  Since
$\norm{S}=M$ and $\norm{S^{-1}}=m^{-1}$,
\[
 \kappa(S)=\frac Mm,
\]
and
\[
 a^2=\ip{Pu}{u},\qquad b^2=\ip{P^{-1}u}{u}.
\]
The Kantorovich inequality \cite{GreubRheinboldt1959} gives
\begin{equation}\label{eq:Kantorovich}
 a^2b^2
 =\ip{Pu}{u}\ip{P^{-1}u}{u}
 \leq\frac{(M^2+m^2)^2}{4M^2m^2}.
\end{equation}
For completeness, its verification here is immediate.  If
$\lambda\in[m^2,M^2]$, then
\[
 \lambda+\frac{m^2M^2}{\lambda}\leq m^2+M^2.
\]
Indeed, this is equivalent to
$(\lambda-m^2)(\lambda-M^2)\leq0$.  Applying the spectral theorem at $u$
gives
\[
 \ip{Pu}{u}+m^2M^2\ip{P^{-1}u}{u}\leq m^2+M^2.
\]
The arithmetic--geometric mean inequality now proves
\eqref{eq:Kantorovich}, and hence
\[
 ab\leq\frac12\left(\kappa (S)+\kappa (S)^{-1}\right),
 \qquad
 t\geq\eta(\kappa (S))\geq\eta(\gamma).
\]
Proposition \ref{prop:nesting}(2) and \eqref{eq:similarity-point} now show that
$\ip{S^{-1}ASu}{u}\in\Omega_{\eta(\gamma)}(A)$.  Since $S$ and $u$ were
arbitrary, the required inclusion follows. 

For the reverse inclusion, set $r=\eta(\gamma)$ and take
$z\in\Omega_r(A)$.  There are unit vectors $x,y$ with
\[
 \ip{x}{y}=r,\qquad z=\frac1r\ip{Ax}{y}.
\]
If $r=1$, then $\gamma=1$, $x=y$, and the choice $S=I$ suffices.  Suppose
$r<1$, put $s=\sqrt{1-r^2}$, and let $v=(x-ry)/s$.  Then
$\{y,v\}$ is orthonormal.  Define a positive matrix $H$ to be the identity
on $\spn\{y,v\}^{\perp}$ and, relative to the ordered basis $\{y,v\}$, to
have matrix
\[
 \begin{pmatrix}
  r&s\\[2pt]
  s&(2-r^2)/r
 \end{pmatrix}.
\]
Thus $Hy=x$.  The displayed matrix has determinant $1$ and trace
$2/r=\gamma+\gamma^{-1}$, so its eigenvalues are
$\gamma$ and $\gamma^{-1}$.  For $S=H^{1/2}$ we therefore have
$\kappa(S)=\gamma$.  Moreover,
\[
 \norm{Sy}^2=\ip{Hy}{y}=\ip{x}{y}=r.
\]
The vector $u=Sy/\sqrt r$ is unit, while
$Su=x/\sqrt r$ and $S^{-1}u=y/\sqrt r$.  Since $S$ is self-adjoint,
\[
 \ip{S^{-1}ASu}{u}
 =\frac1r\ip{S^{-1}Ax}{Sy}
 =\frac1r\ip{Ax}{y}
 =z.
\]
Hence $z\in W(S^{-1}AS)$ for some $S$ with
$\kappa(S)=\gamma$.
\end{proof}

For $0<r\leq1$, define
\begin{equation}\label{eq:chi}
 \chi(r)=\frac{1+\sqrt{1-r^2}}{r}.
\end{equation}
This is the larger positive solution of $\eta(\gamma)=r$; equivalently,
\begin{equation}\label{eq:chi-identities}
 \chi(r)+\chi(r)^{-1}=\frac2r,
 \qquad
 \frac1{\chi(r)}=\frac{r}{1+\sqrt{1-r^2}}.
\end{equation}

\begin{corollary}\label{cor:similarity-interior}
Let $0<r\leq1$.
\begin{enumerate}
\item
\[
 \Omega_r(A)
 =
 \bigcup_{\substack{S\in\Mn\ {\rm invertible}\\
                    \kappa(S)\leq\chi(r)}}
 W(S^{-1}AS).
\]
\item If $A$ is not scalar, $0<r<1$, and $S$ is invertible with
$\kappa(S)<\chi(r)$, then
\[
 W(S^{-1}AS)\Subset\Int\Omega_r(A).
\]
\end{enumerate}
Here $K\Subset U$ means that the compact set $K$ is contained in the open
set $U$.
\end{corollary}

\begin{proof}
\begin{enumerate}
\item
Apply \Cref{thm:similarity-envelope} with
$\gamma=\chi(r)$, for which $\eta(\gamma)=r$.

\item
The strict inequality gives $\eta(\kappa(S))>r$.
By the previous theorem and the strict part of Proposition
\ref{prop:nesting}(2),
\[
 W(S^{-1}AS)
 \subseteq\Omega_{\eta(\kappa(S))}(A)
 \subseteq\Int\Omega_r(A).
\]
The numerical range is compact, which proves the compact containment.
\end{enumerate}
\end{proof}

\section{Extremal pairs and abstract transfer}\label{3}

Throughout this section, let $D\subset\C$ be a bounded convex domain and
let $A\in M_n(\C)$ satisfy $\sigma(A)\subset D$.  We denote by
$\mathcal M_D(A)$ the optimal spectral constant of $\overline D$ for $A$.

Denote
\[
 \mathcal A(\overline D)
 =
 \{f\in C(\overline D): f \text{ is holomorphic in }D\}.
\]
Since $\C\setminus\overline D$ is connected, Mergelyan's theorem
implies that the polynomials are uniformly dense in
$\mathcal A(\overline D)$. Consequently,
\begin{equation}\label{eq:HD-constant}
 \mathcal M_D(A)
 =
 \sup\bigl\{\norm{f(A)}:
 f\in\mathcal A(\overline D),\
 \norm{f}_{\overline D}\leq1\bigr\}.
\end{equation}

The following standard extremal-pair lemma combines Crouzeix's
extremal-function theorem \cite[Theorem~2.1]{Crouzeix2004} with the
orthogonality result in
\cite[Theorem~4.1 and Remark~4.2]{BickelEtAl2020}.

\begin{lemma}\label{lem:extremal-pair}
Let $D\subset\C$ be a bounded convex domain with $\sigma(A)\subset D$.
If $\mathcal M_D(A)>1$, then there exist a function $f$, holomorphic in
$D$ and continuous on $\overline D$, and unit vectors $x,y\in\C^n$ such
that
\[
 \norm{f}_{\overline D}=1,\qquad
 f(A)x=\mathcal M_D(A)y,\qquad
 \ip{x}{y}=0.
\]
\end{lemma}

\begin{theorem}
\label{thm:similarity-amplification}
Suppose that $\mathcal M_D(A)>1$.  For every $\gamma\geq1$ there is a
positive invertible $S$ with $\kappa(S)=\gamma$ such that
\begin{equation}\label{eq:sharp-amplification}
 \mathcal M_D(S^{-1}AS)=\gamma\mathcal M_D(A).
\end{equation}
Consequently,
\begin{equation}\label{eq:max-amplification}
 \max_{\kappa(R)\leq\gamma}\mathcal M_D(R^{-1}AR)
 =\gamma\mathcal M_D(A).
\end{equation}
\end{theorem}

\begin{proof}
Take the extremal data $f,x,y$ from
the previous lemma.  Define the positive matrix $S$ by
\[
 Sx=\sqrt\gamma\,x,\qquad
 Sy=\frac1{\sqrt\gamma}\,y,\qquad
 S|_{\{x,y\}^{\perp}}=I.
\]
Then $\kappa(S)=\gamma$.  If $B=S^{-1}AS$, similarity invariance of the
holomorphic functional calculus gives
\[
 \mathcal M_D(B)\leq\kappa(S)\mathcal M_D(A)=\gamma \mathcal M_D(A).
\]
For the extremal function $f$, however,
\[
 f(B)x
 =S^{-1}f(A)Sx
 =\sqrt\gamma\,\mathcal M_D(A) S^{-1}y
 =\gamma \mathcal M_D(A) y.
\]
Thus $\mathcal M_D(B)\geq\gamma \mathcal M_D(A)$, proving
\eqref{eq:sharp-amplification}.  The same upper bound applies to every
$R$ with $\kappa(R)\leq\gamma$, while this $S$ attains it; hence
\eqref{eq:max-amplification} follows.
\end{proof}

\begin{theorem}\label{thm:abstract-transfer}
Suppose that $\mathcal K(B)$ is a $C$-spectral set for every
$B\in M_n(\C)$. Let $A\in M_n(\C)$, and let $D\subset\C$ be a bounded
convex domain containing $\sigma(A)$. If $\gamma>1$ and
\begin{equation}\label{eq:transfer-hypothesis}
 \mathcal K(S^{-1}AS)\Subset D
 \qquad\text{whenever $S$ is invertible and }\kappa(S)<\gamma,
\end{equation}
then
\[
 \mathcal M_D(A)
 \leq
 \max\left\{1,\frac{C}{\gamma}\right\}.
\]
\end{theorem}

\begin{proof}
If $\mathcal M_D(A)=1$, the conclusion is immediate. Suppose that
$\mathcal M_D(A)>1$, and fix $1<\lambda<\gamma$. By Theorem~3.2, there
is a positive invertible matrix $S_\lambda$ such that
\[
 \kappa(S_\lambda)=\lambda,
 \qquad
 \mathcal M_D(S_\lambda^{-1}AS_\lambda)
 =
 \lambda\mathcal M_D(A).
\]
Since $\lambda<\gamma$, the hypothesis gives
\[
 \mathcal K(S_\lambda^{-1}AS_\lambda)\Subset D.
\]
Consequently, every $f\in\mathcal A(\overline D)$ is holomorphic on a
neighbourhood of $\mathcal K(S_\lambda^{-1}AS_\lambda)$. Since this set
is a $C$-spectral set for $S_\lambda^{-1}AS_\lambda$, we have
\[
 \norm{f(S_\lambda^{-1}AS_\lambda)}
 \leq
 C\norm{f}_{\mathcal K(S_\lambda^{-1}AS_\lambda)}
 \leq
 C\norm{f}_{\overline D}.
\]
Taking the supremum over all
$f\in\mathcal A(\overline D)$ with
$\norm{f}_{\overline D}\leq1$ gives
\[
 \lambda\mathcal M_D(A)
 =
 \mathcal M_D(S_\lambda^{-1}AS_\lambda)
 \leq C.
\]
Letting $\lambda\uparrow\gamma$ yields
\[
 \mathcal M_D(A)\leq\frac{C}{\gamma}.
\]
\end{proof}

\section{Transfer to scaled \texorpdfstring{$q$}{q}-numerical ranges}\label{4}

% We now apply \Cref{thm:abstract-transfer}.  It is useful to formulate the
% result for an arbitrary admissible classical Crouzeix constant.

% \begin{theorem}\label{thm:q-transfer}
% Fix \(n\geq2\), and let \(C\geq1\) be such that \(W(B)\) is a
% \(C\)-spectral set for every \(B\in M_n(\C)\). Then, for every
% \(A\in M_n(\C)\) and every \(q\in\C\) with \(0<|q|\leq1\),
% the set \(\Omega_q(A)\) is a
% \[
%  \max\left\{
%   1,\frac{C|q|}{1+\sqrt{1-|q|^2}}
%  \right\}\text{-spectral set}
% \]
% for \(A\).
% \end{theorem}

% \begin{proof}
% Set \(r=|q|\). By Proposition~2.1(1),
% \(\Omega_q(A)=\Omega_r(A)\). If \(r=1\), then
% \(\Omega_r(A)=W(A)\), and the result follows directly from the
% hypothesis. If \(A=\lambda I\), then
% \(\Omega_r(A)=\{\lambda\}\), which is a \(1\)-spectral set for \(A\).
% We may therefore assume that \(0<r<1\) and that \(A\) is non-scalar.

% Set
% \[
%  D=\operatorname{int}\Omega_r(A).
% \]
% By Proposition~2.1(3), \(\sigma(A)\subset D\). Moreover,
% Corollary~2.3(2) gives
% \[
%  W(S^{-1}AS)\Subset D
%  \qquad\text{whenever}\qquad
%  \kappa(S)<\chi(r).
% \]
% Applying Theorem~3.3 with
% \[
%  \mathcal K(B)=W(B),
%  \qquad
%  \gamma=\chi(r),
% \]
% shows that \(\overline D\) is a
% \[
%  \max\left\{1,\frac{C}{\chi(r)}\right\}\text{-spectral set}
% \]
% for \(A\). Since
% \[
%  \overline D=\Omega_r(A)
%  \qquad\text{and}\qquad
%  \frac{1}{\chi(r)}
%  =\frac{r}{1+\sqrt{1-r^2}},
% \]
% the result follows.
% \end{proof}

For $0<|q|\leq1$, let $\mathfrak C_q$ denote the least $C\geq1$ such
that $\Omega_q(A)$ is a $C$-spectral set for every $n\geq2$ and every
$A\in M_n(\C)$.

The following theorem determines \(\mathfrak C_q\) exactly.

\begin{theorem}\label{thm:q-transfer}
Let $0<|q|\leq1$. Then
\[
 \mathfrak C_q
 =
 \max\left\{
  1,\frac{2|q|}{1+\sqrt{1-|q|^2}}
 \right\}.
\]
\end{theorem}

\begin{proof}
For the upper bound, fix \(n\geq2\) and \(A\in M_n(\C)\). By
Proposition~2.1(1), and since the cases \(\abs{q}=1\) and
\(A=\lambda I\) are immediate, we may assume that
\(0<\abs{q}<1\) and \(A\neq\lambda I\).

By Proposition \ref{prop:nesting}(3) and Corollary \ref{cor:similarity-interior}(2),
\[
 \sigma(A)\subset\Omega_{\abs{q}}(A)^\circ
\]
and
\[
 W(S^{-1}AS)\Subset\Omega_{\abs{q}}(A)^\circ
 \qquad\text{whenever}\qquad
 \kappa(S)<\chi(\abs{q}).
\]
Since, as previously mentioned,  \(W(B)\) is a \(2\)-spectral set
for every square matrix \(B\) \cite{Jin}, Theorem \ref{thm:abstract-transfer}, applied to
\(\Omega_{\abs{q}}(A)^\circ\) with
\[
 \mathcal K(B)=W(B),
 \qquad
 C=2,
 \qquad
 \gamma=\chi(\abs{q}),
\]
(where $\chi$ is defined as in \eqref{eq:chi}) shows that
\[
 \Omega_q(A)=\Omega_{\abs{q}}(A)
 =\overline{\Omega_{\abs{q}}(A)^\circ}
\]
is a
\[
 \max\left\{
  1,\frac{2}{\chi(\abs{q})}
 \right\}
 =
 \max\left\{
  1,\frac{2\abs{q}}{1+\sqrt{1-\abs{q}^2}}
 \right\}\text{-spectral set}
\]
for \(A\). Hence
\[
 \mathfrak C_q
 \leq
 \max\left\{
  1,\frac{2\abs{q}}{1+\sqrt{1-\abs{q}^2}}
 \right\}.
\]

For the reverse inequality, take
\[
 N=\begin{pmatrix}0&2\\0&0\end{pmatrix}.
\]
Proposition \ref{prop:nesting} (1) and the rank-one formula
\cite[Theorem~2.1]{DencicEtAl2025} give
\[
 \max_{z\in\Omega_q(N)}|z|
 =
 \frac{1+\sqrt{1-\abs{q}^2}}{\abs{q}}.
\]
Taking \(p(z)=z\) and using \(\norm{N}=2\), we obtain
\[
 \mathfrak C_q
 \geq
 \frac{\norm{p(N)}}{\norm{p}_{\Omega_q(N)}}
 =
 \frac{2\abs{q}}{1+\sqrt{1-\abs{q}^2}}.
\]
Since every spectral constant is at least \(1\), the reverse inequality
follows, and hence
\[
 \mathfrak C_q
 =
 \max\left\{
  1,\frac{2\abs{q}}{1+\sqrt{1-\abs{q}^2}}
 \right\}.
\]
\end{proof}

\end{document}